\documentclass[a4paper,12pt]{article}
\usepackage[centertags]{amsmath}
\usepackage{amsfonts}
\usepackage{amssymb}
\usepackage{amsthm}
\usepackage{amsmath}
\usepackage{dsfont}
\usepackage{graphicx}
\usepackage{tikz, subfigure}
\usepackage{pstricks-add}
\usepackage{verbatim}

\usepackage[latin1]{inputenc}
\usepackage{tikz}
\definecolor {processblue}{cmyk}{0.96,0,0,0}
\usepackage{amsfonts,graphicx,amsmath,amssymb,hyperref,color}
\usepackage[english]{babel}
\usepackage{authblk}

\newtheorem{theorem}{Theorem}[section]

\newtheorem{Lemma}{Lemma}[section]
\newtheorem{Corollary}{Corollary}[section]

\newtheorem{Example}{Example}[section]

\newtheorem{Remark}{Remark}[section]

\AtEndDocument{\bigskip{\footnotesize%
  \textsc{Department of Mathematics, Shanghai Key Laboratory of PMMP, East China Normal University, Shanghai 200241,
People's Republic of China} \par
  \textit{E-mail address:} \texttt{	
chenxiu1216@163.com.} \par
}}
\AtEndDocument{\bigskip{\footnotesize%
  \textsc{Department of Mathematics, Utrecht University, Fac Wiskunde en informatica and MRI, Budapestlaan 6, P.O. Box 80.000, 3508 TA Utrecht, The Netherlands} \par
  \textit{E-mail address:} \texttt{K.Jiang1@uu.nl.} \par
}}
\AtEndDocument{\bigskip{\footnotesize%
  \textsc{Department of Mathematics, Shanghai Key Laboratory of PMMP, East China Normal University, Shanghai 200241,
People's Republic of China} \par
  \textit{E-mail address:} \texttt{wxli@math.ecnu.edu.cn.} \par
}}

\usepackage{lipsum}

\makeatletter
\newcommand*{\rom}[1]{\expandafter\@slowromancap\romannumeral #1@}
\makeatother
\begin{document}
\title{xxxx}
\date{}
 \title{Lipschitz equivalence of a class of self-similar sets }
\author{Xiu Chen, Kan Jiang\thanks{Kan Jiang is the corresponding author } and Wenxia Li}
\maketitle{}
\maketitle 
\begin{abstract}
We consider a class of homogeneous self-similar sets with complete overlaps and give a sufficient condition for the Lipschitz equivalence between
members in this class.
\bigskip


\end{abstract}

\section{Introduction}

Let $(X_i, d_i), i=1,2$ be metric spaces.  For  nonempty sets $A_i\subseteq X_i$
we say they are Lipschitz equivalent, denoted by $A_1\simeq A_2$, if there exists a bijection $\phi :A_1 \to A_2$ and a
constant $c>0$ such that
$$
c^{-1}d_1(x,y)\leq d_2(\phi (x), \phi (y))\leq c d_1(x, y)\;\;\textrm{for any}\;\; x, y\in A_1.
$$
 Lipschitz equivalence can be used to classify fractal sets.
Since late 80's  many works have been devoted to
the study of Lipschitz equivalence
(see \cite{DavidSemmes, FaMa, FaMa1, Xi, LLM,  lk, Matt, Raohui, WX, WZD, xix, Xi1} and references therein).
 An effective method, to our knowledge, was first employed in \cite{Raohui} for establishing a bi-Lipschitz mapping between
 the $\{1,4,5\}$-Cantor set and the $\{1,3,5\}$-Cantor set, the main idea of which is  to show
 these two self-similar sets  to have same graph-directed structure satisfying the strong separation
 condition.
   A sufficient condition was given in \cite[Theorem 2.11]{KarmaKan} to judge whether or not a self-similar set has a
 graph-directed structure
 satisfying the open set condition or even the strong separation condition.

In the present we consider the  homogeneous iterated function system (IFS)
   $\{f_{i}(x)=\lambda x+a_i: 1\leq i\leq m \}$ where $x, a_i\in \mathbb R$, $\lambda \in (0,1)$ and the integer $m\geq 3$.  For a vector  $ (k_{1},\ldots,k_{n})$ of integers with $k_1>k_2>\cdots >k_{n}\geq 2$ , let ${\mathbb A}_{k_{1},\ldots,k_{n}}$ be the collection of translations ${\bf a}=(a_1, a_2, \cdots , a_m)$  satisfying the following conditions (I) (II) and (III):

  (I) $ 0=a_1<a_2<\cdots <a_m=1-\lambda $;

  (II) Any three intervals in $\{f_i([0,1]): 1\leq i\leq m\}$ do not intersect. $|f_i([0,1])\cap f_j([0,1])| \in \{\lambda ^{k_{1}},\cdots ,\lambda ^{k_{n}}\}$ whenever $f_i([0,1])\cap f_j([0,1])\neq \emptyset $ with $i<j$, where by $|J|$ we denote the length of an interval $J$;

  (III) Either $f_1([0,1])\cap f_j([0,1])= \emptyset $ for all $j>1$, or $f_m([0,1])\cap f_j([0,1])= \emptyset $ for all $j<m$.

  From (I) and (II) it follows that when $|f_i([0,1])\cap f_j([0,1])|=\lambda ^{k_{\ell}}$ with $i<j$, then $j=i+1$ and $f_i\circ f_m^{k_\ell -1}=f_j\circ f_1^{k_\ell -1}$. Throughout this paper,
  $f^i$ stands for the $i$-th iteration of map $f$ for $i\in \mathbb N$. In particular, $f^0$ stands for the identity.

  For a translation ${\bf a }=(a_1, a_2, \cdots , a_m)\in {\mathbb A}_{k_{1},\ldots,k_{n}}$, Let
  $$
  \gamma _\ell ({\bf a})=\left \{1\leq i\leq m: |f_i([0,1])\cap f_{i+1}([0,1])|=\lambda ^{k_{\ell}}\right \}\;\;\textrm{for}\;\; 1\leq \ell \leq n.
  $$

  It is well known that for each ${\bf a }=(a_1, a_2, \cdots , a_m)\in {\mathbb A}_{k_{1},\ldots,k_{n}}$, there exists a unique nonempty compact set  $K_{\bf a}$ such that $K_{\bf a}=\bigcup _{1\leq i\leq m}f_i(K_{\bf a})$ (see \cite{Hutchinson}). The set $K_{\bf a}$ is called  the self-similar set  generated by the IFS $\{f_{i}(x)=\lambda x+a_i: 1\leq i\leq m \}$. Let   $\#S$  denote   the number of elements of $S$. In this paper we obtain

  \begin{theorem}\label{Main}
  For ${\bf a}, {\bf b}\in {\mathbb A}_{k_{1},\ldots,k_{n}}$ we have $K_{\bf a}\simeq K_{\bf b}$ if  $\#\gamma _\ell ({\bf a})=\#\gamma _\ell ({\bf b})$ for $1\leq \ell \leq n$.
  \end{theorem}

  It is clear that $\dim _HE=\dim _HF$ if $E\simeq F$. Thus we have
\setcounter{Corollary}{1}
  \begin{Corollary}
For ${\bf a}, {\bf b}\in {\mathbb A}_{k_1, \cdots , k_n}$ we have $\dim _HK_{\bf a}=\dim _H K_{\bf b}$
 if  $\#\gamma _\ell ({\bf a})=\#\gamma _\ell ({\bf b})$ for $1\leq \ell \leq n$.
\end{Corollary}

We prove Theorem \ref{Main}  and give some examples in the next section.

\section{Proof of Theorem \ref{Main}}

Before proving Theorem \ref{Main}, let us recall the graph-directed self-similar set (see \cite{mauwil}).
Let ${\mathcal G}=(V,E)$ be a directed graph where $V$ is a finite set of  vertexes and $E$ is a finite set
 of directed edges.
Assume that for any $u\in V$ there is at least one edge in $E$ starting from $u$.
For an $e\in E$, let  $f_e:\mathbb{R}^n\to\mathbb{R}^n $ be a similitude
 with  ratio $\rho_e\in(0,1)$, namely
$$
|f_e(x)-f_e(y)|=\rho_e|x-y|
\;\;\textrm{for any}\;\;    x,y\in \mathbb{R}^n.
$$
 Then there exist unique nonempty compact sets $\{F_u: u\in V\}$ such that
\begin{equation}\label{gdsss}
F_u=\bigcup_{v\in V}\bigcup_{e\in E_{u,v}}f_e(F_v)\;\; \textrm{for all }\;\; u\in V,
\end{equation}
where $E_{u, v}$ is the set of directed edges starting from $u$ and ending at $v$. The compact sets  $\{F_u: u\in V\}$ in
(\ref{gdsss})
is called the graph-directed self-similar sets generated by $\{V, E, \{f_e: e\in E\}\}$.
In addition, $\{F_u: u\in V\}$ is said to satisfy the
strong separation condition if the sets in the right side of (\ref{gdsss}) are pairwise disjoint. An easy-to-prove result
on the Lipschitz equivalence between two graph-directed self-similar sets is as follows ( also see \cite{Raohui} ).
\begin{Lemma}\label{gdlip}
Let $\{F_u: u\in V\}$  and $\{G_u: u\in V\}$ be
the graph-directed self-similar sets generated by $\{V, E, \{f_e: e\in E\}\}$ and $\{V, E, \{g_e: e\in E\}\}$, respectively.
Suppose that for each $e\in E$ the similitudes $f_e$ and $g_e$ have the same ratio $\rho _e$, and
both $\{F_u: u\in V\}$  and $\{G_u: u\in V\}$ satisfy the strong separation condition. Then for each $u\in V$, we have
$F_u\simeq G_u$.
\end{Lemma}
\begin{proof}
Fix a $u\in V$. We denote by $E_v$ the set of directed edges starting from $v$ for $v\in V$. For a directed edge $e\in E$ we denote its initial and ending points by $e^-$ and $e^+$, respectively.  Let
$$
c_*=\min _{v\in V}\min \left \{d(f_{e_*}(F_{e_*^+}), f_{e_{**}}(F_{e_{**}^+})),\; d(g_{e_*}(G_{e_*^+}), g_{e_{**}}(G_{e_{**}^+})): e_*\ne e_{**}\in E_v\right \}
$$
and
$$
c^*=\max \left\{\textrm{diameter of the set }\; \bigcup _{v\in V}F_v,\; \textrm{ diameter of the set }\; \bigcup _{v\in V}G_v\right \}.
$$
Then $c_*, c^*>0$.
 An infinite
  sequence of directed edges
$e_1e_2\cdots $ is called admissible if $e_i^+$ coincides with $e^-_{i+1}$ for all $i\in \mathbb N$. Let
$$
\Sigma _u=\{e_1e_2\cdots : \; e_1e_2\cdots \;\textrm{is admissible with}\;\; e^-_1=u\}.
$$
Then the maps
$$
\Pi_F(e_1e_2\cdots )=\bigcap _{i=1}^\infty f_{e_1}\circ \cdots \circ f_{e_i}(F_{e^+_i})\;\textrm{and}\; \Pi_G(e_1e_2\cdots )=\bigcap _{i=1}^\infty g_{e_1}\circ \cdots \circ g_{e_i}(G_{e^+_i})
$$
are bijections  between $\Sigma _u$ and $F_u$, and between $\Sigma _u$ and $G_u$ respectively.
We shall check the bijection $\Pi_G\circ \Pi^{-1}_F$
is bi-Lipschitz. Let $x,y\in F_u$ with $x\ne y$. Then there exist unique
$(e_i), (s_i)\in \Sigma _u$  such that $x=\Pi_F(e_1e_2\cdots ), y=\Pi_F(s_1s_2\cdots )$. Let $\ell $ be the smallest integer such that
$e_\ell\ne s_\ell $. Then we have $e_\ell^-= s_\ell ^-$ and $e_\ell^+\ne s_\ell ^+$ because of $e_1^-=s_1^-=u$. This implies that $x=f_{e_1}\circ \cdots \circ f_{e_{\ell -1}}(x^*)$ and $y=f_{e_1}\circ \cdots \circ f_{e_{\ell -1}}(y^*)$ with $x^*\in f_{e_\ell}( F_{e_\ell^+}), y^*\in f_{s_\ell }( F_{s_\ell ^+})$.
So
$$
c_*\prod _{i=1}^{\ell -1}\rho _{e_i}\leq |x-y|\leq c^* \prod _{i=1}^{\ell -1}\rho _{e_i},
$$
Note that
$$
\Pi_G\circ \Pi^{-1}_F(x)=\Pi_G(e_1e_2\cdots )\;\;\textrm{and}\;\;\Pi_G\circ \Pi^{-1}_F(y)=\Pi_G(s_1s_2\cdots ),
$$
which implies, by the same argument as above, that
$$
c_*\prod _{i=1}^{\ell -1}\rho _{e_i}\leq |\Pi_G\circ \Pi^{-1}_F(x)-\Pi_G\circ \Pi^{-1}_F(y)|\leq c^* \prod _{i=1}^{\ell -1}\rho _{e_i}.
$$
Therefore, $\Pi_G\circ \Pi^{-1}_F$
is bi-Lipschitz.
\end{proof}

Now we are ready to prove Theorem \ref{Main}.

\begin{proof}[Proof of Theorem \ref{Main}]
By $\{f_i: 1\leq i\leq m\}$ and $\{g_i: 1\leq i\leq m\}$ we denote the iterated function systems corresponding to translations ${\bf a}=(a_1, a_2, \cdots , a_m)$ and
${\bf b}=(b_1, b_2, \cdots , b_m)$, respectively.

Without loss of generality we may assume that $f_m([0,1])\cap f_j([0,1])= \emptyset $ for all $j<m$, and that
$g_m([0,1])\cap g_j([0,1])= \emptyset $ for all $j<m$ in condition (III). To understand it, one only need to notice the following facts:
we have that $K_{\bf c}=1-K_{\bf a}\simeq K_{\bf a}$ for the translation
${\bf c}=(1-\lambda -a_m, 1-\lambda -a_{m-1}, \cdots, 1-\lambda -a_{2}, 1-\lambda -a_{1})\in {\mathbb A}_{k_1, \cdots , k_n}$, and $\#\gamma _\ell ({\bf c})=\#\gamma _\ell ({\bf a})$ for $1\leq \ell \leq n$. Let
 $$
  \gamma _{n+1} ({\bf a})=\{1, \cdots , m-1\}\setminus \bigcup _{\ell=1}^n \gamma _\ell ({\bf a}).
  $$
   We relabel the elements of $\gamma _\ell ({\bf a})$ in its increasing order by
  digits $\{1+\sum _{j=1}^{\ell -1}\#\gamma _{j} ({\bf a}), 2+\sum _{j=1}^{\ell -1}\#\gamma _{j} ({\bf a}),\cdots , \sum _{j=1}^{\ell }\#\gamma _{j} ({\bf a})\}$ with $\sum _{j=1}^{0}\#\gamma _{j} ({\bf a})=0$ and $1\leq \ell \leq n+1$. By $h(\cdot )$ we denote this relabeling. Thus
  $h(j)\in \{1+\sum _{j=1}^{\ell -1}\#\gamma _{j} ({\bf a}), 2+\sum _{j=1}^{\ell -1}\#\gamma _{j} ({\bf a}),\cdots , \sum _{j=1}^{\ell }\#\gamma _{j} ({\bf a})\}$
  for $j\in \gamma _\ell ({\bf a})$.

  We partition $K_{\bf a}$ into $m+k_1-2$ pairwise disjoint nonempty compact sets, denoted by $K_1, \cdots , K_{m+k_1-2}$. The first $m-1$ members of them are defined by
  \begin{equation}\label{firsmminusone}
  K_{h(j)}=\left \{
  \begin{array}{ll}
  f_j(K_{\bf a})\setminus f_j\circ f^{k_\ell -1}_m(K_{\bf a})\;\;&\;\;\textrm{for}\;\; j\in \gamma _{\ell }({\bf a}),\; 1\leq \ell \leq n\\
  f_j(K_{\bf a}) ;&\;\;\textrm{for}\;\; j\in \gamma _{n+1 }({\bf a}).
  \end{array}
  \right.
  \end{equation}
  The later $k_1-1$ members of $K_i$s are defined by
  \begin{equation}\label{latermem}
  \left \{
  \begin{array}{ll}
  K_{m+k_1-2}=f_m^{k_1-1}(K_{\bf a})   & \;\; \\
  K_{m+t}=f_m^{t+1}(K_{\bf a})\setminus f^{t+2}_m(K_{\bf a}) \;\; \textrm{for}\;\; 0\leq t< k_1-2.
  \end{array}
  \right.
  \end{equation}
  From (\ref{firsmminusone}) and (\ref{latermem}) it follows that $K_{\bf a}=\bigcup _{i=1}^{m+k_1-2}K_i$ with disjoint union. It is important
  to notice that for $1\leq \ell \leq n$
  \begin{equation*}
  f_m(K_{\bf a})=\left \{
  \begin{array}{ll}
  K_m\cup K_{m+1}\cup \cdots \cup K_{m+k_\ell -3}\cup f_m^{k_\ell -1}(K_{\bf a})&\textrm{with disjoint union for}\;\; k_\ell \geq 3\\
   f_m(K_{\bf a}) \;\;&\textrm{ for}\; k_\ell =2.
  \end{array}
  \right.
  \end{equation*}
Thus we have for $j\in \gamma _{\ell }({\bf a})$ with $1\leq \ell \leq n$
\begin{equation*}
\begin{split}
K_{h(j)}&=f_j(K_{\bf a})\setminus f_j\circ f^{k_\ell -1}_m(K_{\bf a})=\left (f_j(K_1\cup \cdots \cup K_{m-1}\cup f_m(K_{\bf a}))\right )
\setminus f_j\circ f^{k_\ell -1}_m(K_{\bf a})\\
&=\bigcup _{i=1}^{m+k_\ell -3}f_j(K_i).
\end{split}
\end{equation*}
It is obvious that for $j\in \gamma _{n+1}({\bf a})$
$$
K_{h(j)}
=\bigcup _{i=1}^{m+k_1 -2}f_j(K_i).
$$
Finally, Note that $K_{\bf a}=K_1\cup \cdots \cup K_{m-1}\cup f_m(K_{\bf a})$ with disjoint union. Thus, for $0\leq t< k_1-2$
 \begin{equation*}
K_{m+t}=
  f_m^{t+1}(K_{\bf a})\setminus f_m^{t+2}(K_{\bf a})=\bigcup _{i=1}^{m-1} f^{t+1}_m(K_i)
  \end{equation*}
and
 \begin{equation*}
K_{m+k_1-2}=\left \{
\begin{array}{ll}
 f_m (K_{m+k_1-3})\cup f_m(K_{m+k_1-2})\;&\textrm{when}\; k_1\geq 3\\
 \bigcup _{i=1}^m f_m(K_i) \;&\textrm{when}\; k_1=2.
\end{array}
\right.
 \end{equation*}

Therefore, $(K_1, \cdots , K_{m+k_{1}-2})$ are graph-directed self-similar sets satisfying the strong separation condition. By the same argument as above
by replacing $f_i$ by $g_i$, one can get pairwise disjoint nonempty compacts $K^*_{i}$ with $K_{\bf b}=\bigcup _{1\leq i\leq {m+k_{1}-2}}K^*_{i}$. The $(K^*_1, \cdots , K^*_{m+k_{1}-2})$ are graph-directed self-similar sets satisfying the strong separation condition and obey the same equations as $K_i$s
 with replacing $f_i$ by $g_i$. Thus $K_i\simeq K^*_i$ for $1\leq i\leq m+k_1-2$, and so $K_{\bf a}\simeq K_{\bf b}$ because of the disjointness of $K_i$s and disjointness of $K^*_i$s.
\end{proof}

\setcounter{Example}{1}
\begin{Example}\label{EX}
Let $0<\lambda<5^{-1}$. Take  ${\bf a}=(0, \lambda(1-\lambda), 2\lambda(1-\lambda), 3\lambda,
1-\lambda)$ and
${\bf b}=(0, \lambda(1-\lambda), 2\lambda, 3\lambda-\lambda^2,
1-\lambda)$. Then one can check that ${\bf a}, {\bf b}\in {\mathbb A}_2$,  $\gamma _1({\bf a})=\{1,2\}$
 and $\gamma _1({\bf b})=\{1, 3\}$. Thus $K_{\bf a}\simeq K_{\bf b}$ by Theorem \ref{Main}.
\end{Example}

The approach presented in this paper can be also applied for higher dimensional case.
\begin{Example}\label{EX2}
Let $0<\lambda<(2-\sqrt{2})/2$. Consider two IFSs $\{f_i: 1\leq i\leq 6\}$ and  $\{g_i: 1\leq i\leq 6\}$ where
\begin{equation*}
\begin{array}{ll}
  f_{1}(x,y)=\lambda (x, y),\;
    &f_{2}(x,y)=\lambda (x, y)+(1-\lambda, 0),\\
  f_{3}(x,y)=\lambda (x, y)+(1-\lambda,  1-\lambda ),\;
    & f_{4}(x,y)=\lambda (x, y)+(0, 1-\lambda),\\
  f_{5}(x,y)=\lambda (x, y)+(\lambda(1-\lambda),  (1-\lambda)^2 ),\,  & f_{6}(x,y)=\lambda (x, y)+(0, (1-\lambda)(1-2\lambda)),
     \end{array}
  \end{equation*}
 and $g_{6}(x,y)=\lambda (x, y)+(\lambda(1-\lambda), \lambda(1-\lambda))$
 with  $ g_i(x,y)=f_i(x,y)$ for $1\leq i\leq 5$. Let $F$ and $G$ be the self-similar sets generated by IFSs $\{f_i: 1\leq i\leq 6\}$ and  $\{g_i: 1\leq i\leq 6\}$, respectively. Then $F\simeq G$.
 \end{Example}
\begin{proof}
 Figure \ref{figtwo} shows locations of squares $f_i([0,1]^2), 1\leq i\leq 6$ and squares $g_i([0,1]^2), 1\leq i\leq 6$.
 Let $F_i=f_i(F)$ for $i=1,2,3,5$, $F_4=f_4(F)\setminus f_4\circ f_2(F)$ and
 $F_6=f_6(F)\setminus f_6\circ f_3(F)$. Then $F_i, 1\leq i\leq 6$, are pairwise disjoint nonempty compact sets such that $F=\bigcup _{1\leq i\leq 6}F_i$ since
 $f_{4}\circ f_2=f_{5}\circ f_4$ and  $f_{6}\circ f_3=f_{5}\circ f_1$.

Thus we have
\begin{equation*}
\left \{
\begin{array}{ll}
F_i= f_i(F_1)\cup f_i(F_2)\cup f_i(F_3)\cup f_i(F_4)\cup f_i(F_5)\cup f_i(F_6)\;\; &\;\textrm{for}\;\; i=1,2,3,5\\
F_4=f_4(F_1)\cup f_4(F_3)\cup f_4(F_4)\cup f_4(F_5)\cup f_4(F_6)\;\; &
\\
F_6=f_6(F_1)\cup f_6(F_2)\cup f_6(F_4)\cup f_6(F_5)\cup f_6(F_6)\;\; &
\end{array}
\right.
\end{equation*}
\begin{figure}
  \includegraphics[width=380pt]{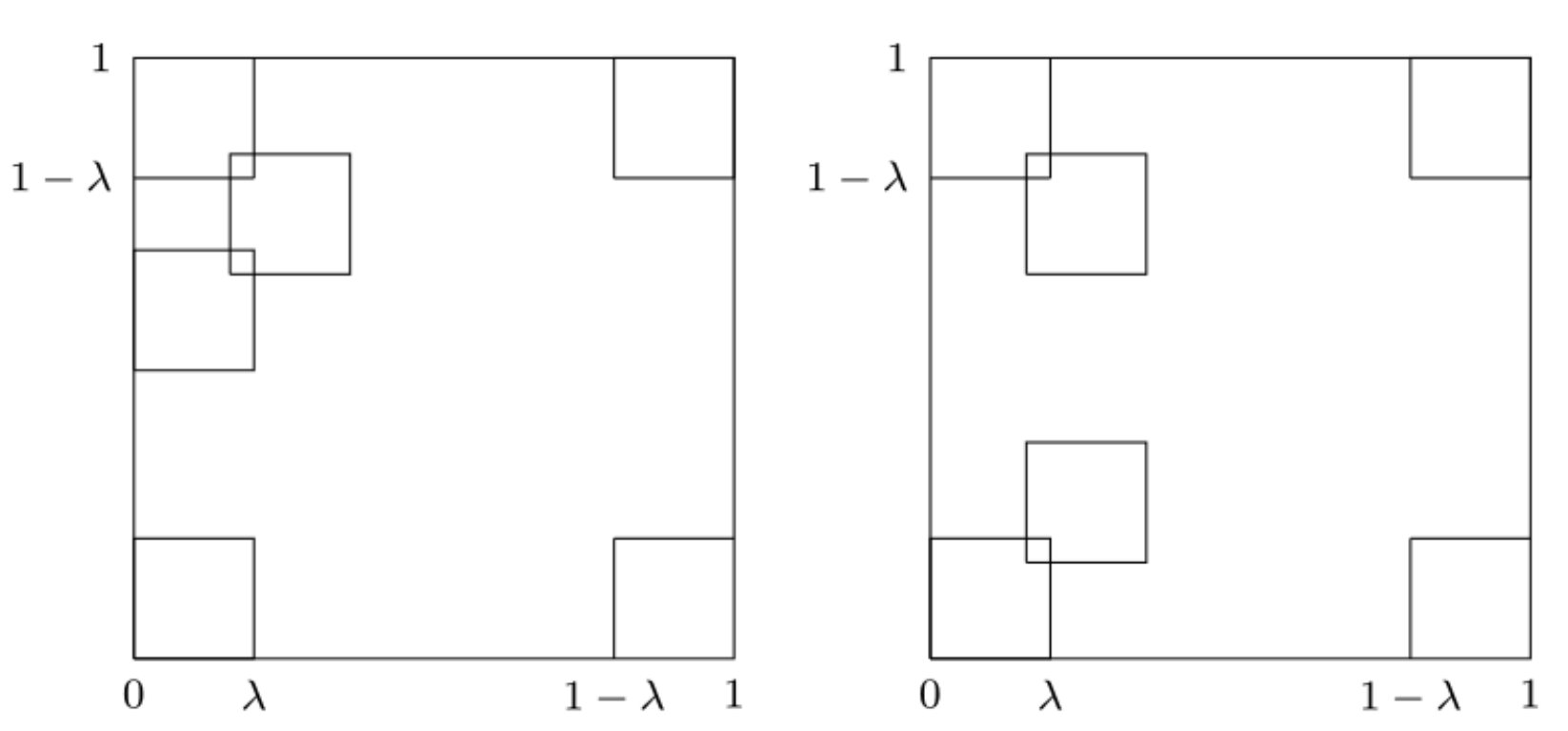}\\
  \caption{Squares $f_i([0,1]^2)$ on the left side and Squares $g_i([0,1]^2)$ on the right side}\label{figtwo}
\end{figure}
By the same way as above let $G_1=g_5(G)$, $G_2=g_2(G)$, $G_3=g_3(G)$, $G_4=g_4(G)\setminus g_4\circ g_2(G)$, $G_5=g_6(G)$ and $G_6=g_1(G)\setminus g_1\circ g_3(G)$.
We have $G_i, 1\leq i\leq 6$, are pairwise disjoint nonempty compact sets with $G=\bigcup _{1\leq i\leq 6}G_i$ and satisfy
\begin{equation*}
\left \{
\begin{array}{l}
G_1= g_5(G_1)\cup g_5(G_2)\cup g_5(G_3)\cup g_5(G_4)\cup g_5(G_5)\cup g_5(G_6)\\
G_2= g_2(G_1)\cup g_2(G_2)\cup g_2(G_3)\cup g_2(G_4)\cup g_2(G_5)\cup g_2(G_6)\\
G_3= g_3(G_1)\cup g_3(G_2)\cup g_3(G_3)\cup g_3(G_4)\cup g_3(G_5)\cup g_3(G_6)\\
G_5= g_6(G_1)\cup g_6(G_2)\cup g_6(G_3)\cup g_6(G_4)\cup g_6(G_5)\cup g_6(G_6)\\
G_4=g_4(G_1)\cup g_4(G_3)\cup g_4(G_4)\cup g_4(G_5)\cup g_4(G_6)
\\
G_6=g_1(G_1)\cup g_1(G_2)\cup g_1(G_4)\cup g_1(G_5)\cup g_1(G_6).
\end{array}
\right.
\end{equation*}
Thus $F\simeq G$ by Lemma \ref{gdlip}.
\end{proof}

\begin{Example}
Let $0<\lambda <\frac{1}{7}$. Let $G$ be the self-similar set generated by the IFS $\{g_i: 1\leq i\leq 6\}$  given  in Example \ref{EX2}. Let $F$ be the self-similar set generated by the IFS $\{f_i: 1\leq i\leq 6\}$ where  $f_{1}(x)=\lambda x$,
  $ f_{2}(x)=\lambda x +2\lambda $,
  $f_{3}(x)=\lambda x+3\lambda-\lambda^2$,
  $ f_{4}(x)=\lambda x +4\lambda-2\lambda^2$,
$
  f_{5}(x)=\lambda x+5\lambda$ and
  $ f_{6}(x)=\lambda x +1-\lambda$. Then $F\simeq G$.
\end{Example}
\begin{proof}
Note that $F^*=F\times \{0\}$ is the self-similar set in ${\mathbb R}^2$ generated by the IFS:
\begin{equation*}
\begin{array}{ll}
  f^*_{1}(x,y)=\lambda (x, y),\;
    &f^*_{2}(x,y)=\lambda (x, y)+(2\lambda, 0),\\
  f^*_{3}(x,y)=\lambda (x, y)+(3\lambda-\lambda^2,  0),\;
    & f^*_{4}(x,y)=\lambda (x, y)+(4\lambda-2\lambda^2, 0),\\
  f^*_{5}(x,y)=\lambda (x, y)+(5\lambda, 0),\,  & f^*_{6}(x,y)=\lambda (x, y)+(1-\lambda, 0).
     \end{array}
  \end{equation*}
By letting $F_1=f^*_3(F^*), F_2=f^*_1(F^*), F_3=f^*_6(F^*),
F_5=f^*_5(F^*)$, $F_6=f^*_2(F^*)\setminus f^*_2\circ f^*_6(F^*), F_4=f^*_4(F^*)\setminus f^*_4\circ f^*_1(F^*)$, one can get $F^*$ has the same graph-directed structure as $G$. Thus we have $G\simeq F^*\simeq F$ by Lemma \ref{gdlip}.
\end{proof}
\setcounter{Remark}{4}
\begin{Remark}
 It is natural to compare our method with the idea used in \cite{Raohui}. On the one hand, our method cannot prove the Lipschitz equivalence between  the $\{1,4,5\}$-Cantor set and the $\{1,3,5\}$-Cantor set. The main difficulty, which is crucial,  is that our idea   only transforms the $\{1,4,5\}$-Cantor set  into a graph-directed self-similar sets with the open set condition rather than the strong separation condition. It is not enough if we only obtain the open set condition. That is why we cannot reprove the main result of \cite{Raohui}.  On the other hand, in terms of the approach of \cite{Raohui}, it seems that we cannot obtain Theorem \ref{Main}.    In brief, these two methods above are independent, i.e. the idea of   \cite{Raohui} is useful when one tackles the self-similar sets with the open set condition, while our method is effective for the self-similar sets with exact overlaps.
\end{Remark}

\section*{Acknoledgment}
The first  author was granted by the China Scholarship Council No. 201606140059.  The second author was granted by the China Scholarship Council No. 201206140003.  The third author  was supported by  NSFC No. 11271137, 11571144, 11671147 and Science and Technology Commission of Shanghai Municipality (STCSM), grant No. 13dz2260400.

\end{document}